\documentclass{amsart}
\usepackage{amssymb}
\usepackage{amsfonts}
\usepackage{amssymb}
\usepackage{amsmath, accents}
\usepackage{amsthm}
\usepackage{pdfpages}
\usepackage{hyperref}
\usepackage{enumerate}
\usepackage{tabularx}
\usepackage{centernot}
\usepackage{mathtools}
\usepackage{stmaryrd}
\usepackage{amsthm,amssymb}
\usepackage{etoolbox}
\usepackage{tikz}
\usepackage{amssymb}
\usetikzlibrary{matrix}
\usepackage{tikz-cd}
\usepackage{tikz}
\usepackage{marginnote}
\definecolor{mygray}{gray}{0.85}
\usepackage[backgroundcolor=mygray,colorinlistoftodos,prependcaption,textsize=small]{todonotes}
\usepackage{xcolor}

\renewcommand{\leq}{\leqslant}
\renewcommand{\geq}{\geqslant}

\makeatletter
\def\subsection{\@startsection{subsection}{3}%
  \z@{.5\linespacing\@plus.7\linespacing}{.3\linespacing}%
  {\bfseries\centering}}
\makeatother

\makeatletter
\def\subsubsection{\@startsection{subsubsection}{3}%
  \z@{.5\linespacing\@plus.7\linespacing}{.3\linespacing}%
  {\centering}}
\makeatother

\makeatletter
\def\myfnt{\ifx\protect\@typeset@protect\expandafter\footnote\else\expandafter\@gobble\fi}
\makeatother

\newtheorem{theorem}{Theorem}[section]

\newtheorem{corollary}[theorem]{Corollary}
\newtheorem{definition}[theorem]{Definition}

\newtheorem{fact}[theorem]{Fact}

\newtheorem{notation}[theorem]{Notation}

\newtheorem{cclaim}[theorem]{Claim}

\newtheorem{definition-proposition}[theorem]{Definition/Proposition}

\newcounter{claimcounter}
\numberwithin{claimcounter}{theorem}

\usepackage{eso-pic}

\newcommand{\mrm}[1]{\mathrm{#1}}

\begin{document}

\begin{abstract} In \cite{1205} we proved that the space of countable torsion-free abelian groups is Borel complete. In this paper we show that our construction from \cite{1205} satisfies several additional properties of interest. We deduce from this that countable torsion-free abelian groups are faithfully Borel complete, in fact, more strongly, we can $\mathfrak{L}_{\omega_1, \omega}$-interpret countable graphs in them. Secondly, we show that the relation of pure embeddability (equiv., elementary embeddability) among countable \mbox{models of $\mrm{Th}(\mathbb{Z}^{(\omega)})$ is a complete analytic quasi-order.}
\end{abstract}

\title[Torsion-free abelian groups are faithfully Borel complete]{Torsion-free abelian groups are faithfully Borel complete and pure embeddability is a complete analytic quasi-order}


\thanks{No. 1248 on Shelah's publication list. Research of the first author was supported by project PRIN 2022 ``Models, sets and classifications", prot. 2022TECZJA and by INdAM Project 2024 (Consolidator grant) ``Groups, Crystals and Classifications''. Research of the second author was partially supported by Israel Science Foundation (ISF) grants no: 1838/19 and 2320/23.}

\author{Gianluca Paolini}
\address{Department of Mathematics ``Giuseppe Peano'', University of Torino, Via Carlo Alberto 10, 10123, Italy.}
\email{gianluca.paolini@unito.it}

\author{Saharon Shelah}
\address{Einstein Institute of Mathematics,  The Hebrew University of Jerusalem, Israel \and Department of Mathematics,  Rutgers University, U.S.A.}

\date{\today}
\maketitle


\section{Introduction}

	In \cite{1205} we showed that the Borel space of countable torsion-free abelian groups ($\mrm{TFAB}_\omega$) is as complex as possible in terms of classification up to isomorphism, resolving a major conjecture of Friedman and Stanley from 1989 (cf. \cite{friedman_and_stanley}). The aim of this paper is to show that our construction from \cite{1205} satisfies several additional properties of interest, which imply stronger anti-classification results for the space $\mrm{TFAB}_\omega$. In particular, we will prove the following (see what follows for a discussion):

	\begin{theorem}\label{first_theorem} $\mrm{TFAB}_\omega$ is a faithfully Borel complete class of structures. Furthermore, we can $\mathfrak{L}_{\omega_1, \omega}$-interpret the space $\mrm{Graphs}_\omega$ (graphs with domain $\omega$) into the space $\mrm{TFAB}_\omega$ (torsion-free abelian groups with domain $\omega$).
\end{theorem}

	\begin{theorem}\label{second_theorem} The pure embeddability relation on $\mrm{TFAB}_\omega$ is a complete analytic quasi-order. In fact, more strongly,  elementary embeddability (equiv., pure embeddability) between countable models of $\mrm{Th}(\mathbb{Z}^{(\omega)})$ is a complete analytic quasi-order.
\end{theorem}

	The property of Borel completeness for the space of countable models of a theory in $\mathfrak{L}_{\omega_1, \omega}$ is probably the most well-known anti-classification property in terms of classification up to isomorphism, as it literally says that the isomorphism relation on such a class reduces in a Borel way the isomorphism relation on countable models of {\em any} theory in $\mathfrak{L}_{\omega_1, \omega}$. But, actually, stronger forms of anti-classification are known in the literature, for example, the fact that countable graphs can be first-order interpreted in countable groups (cf. e.g. \cite{mekler}) is widely agreed to be a much stronger result than the Borel completeness of the space of countable groups. This line of thought was already addressed by Friedman and Stanley in their seminal paper on Borel reducibility \cite{friedman_and_stanley}, in fact, abstracting from the model theoretic notion of interpretability, they\footnote{The use of the term {\em faithful} to denote this property was introduced only later, cf. \cite[pg.~300]{gao}.} introduced the following strengthening of Borel completeness:
	
	\begin{definition} Let $\mathbf{K}_\omega$ be the Borel space of models with domain $\omega$ of a $\mathfrak{L}_{\omega_1, \omega}$-theory. The space $\mathbf{K}_\omega$ is said to be {\em faithfully} Borel complete if there is a Borel reduction $\mathbf{F}$ from $\mrm{Graph}_\omega$ (graphs with domain $\omega$) into $\mathbf{K}_\omega$ such that for any invariant Borel subset $X$ of $\mrm{Graph}_\omega$ the closure under isomorphism of the image of $X$ under $\mathbf{F}$ is Borel.
\end{definition}

	It is well-known that countable groups are faithfully Borel complete and that furthermore there is a first-order interpretation of countable graphs into countable groups. On the other hand, any first-order theory of abelian groups is known to be stable and so we cannot expect to have a first-order interpretation of countable graphs in countable abelian groups. Our Theorem~\ref{first_theorem} is then the next best possible result in this respect; additionally the interpretation can also be taken to be with respect to very simple formulas, see Notation~\ref{the_pure_formulas_notaiton} for details. In \cite{friedman_and_stanley} one of the main motivations for the introduction of the notion of faithful Borel completeness was that whenever this property holds for $T$, then the full Vaught's conjecture reduces to the Vaught's conjecture for $\mathfrak{L}_{\omega_1, \omega}$-theories extending $T$, in particular we deduce:
	
	\begin{corollary} Vaught's conjecture is equivalent to Vaught's conjecture for $\mathfrak{L}_{\omega_1, \omega}$-theories of torsion-free abelian groups of infinite rank.
\end{corollary}  

	We now comment on Theorem~\ref{second_theorem}. In recent years, descriptive set theorists have been paying attention to other equivalence relations or quasi-orders among classes of countable structures. In particular, among many other interesting results, in \cite{LouRos} it was shown that the embeddability relation between countable graphs is a complete analytic quasi-order, and so the relation of bi-embeddability among countable graphs is a complete analytic equivalence relation. Despite this, not much seems to be known in terms of analysis of the relation of elementary embeddability, apart from reference \cite{rossenger}, where it is shown that this relation when considered between countable graphs is a complete analytic quasi-order. In particular, a careful analysis of the complexity of the relation of elementary embeddability between the countable models of familiar complete first-order theories does not seem to be addressed in the literature (notice that on the other hand in terms of complexity of isomorphism the situation is much different, as e.g. for any complete first-order theory $T$ of Boolean algebras we know the exact complexity of the relation of isomorphism between the countable models of $T$, see \cite{camerlo}). In this respect our Theorem \ref{second_theorem} seems to be particularly relevant, and we hope that it will inspire further research on the topic. Finally, we want to mention that in \cite{calderoni} it was proved that the embeddability relation between countable abelian \mbox{groups is also a complete analytic quasi-order.}
	
%
%
%

	Some words of explanations on the structure of the paper are in order. In Section 2 we overview the construction from \cite{1205} referring to \cite{1205} for details. In Sections~3 and 4 we prove Theorems~\ref{first_theorem} and \ref{second_theorem}. We follow the same notation of \cite{1205}, so we invite the reader to refer to \cite{1205} for unexplained notation. We only recall:
	
	\begin{definition}\label{def_pure} Let $H \leq G$ be groups, we say that $H$ is pure in $G$, denoted by $H \leq_* G$, when if $h \in H$, $0 < n < \omega$, $g \in G$ and (in additive notation) $G \models ng = h$, then there is $h' \in H$ s.t. $H \models nh' = h$. Given $S \subseteq G$ we denote by $\langle S \rangle^*_S$ the pure subgroup generated by $S$ (the intersection of all the \mbox{pure subgroups of $G$ containing~$S$).}
\end{definition}

\section{Overview of the construction from \cite{1205}}\label{S3}

	The construction from \cite{1205} consists of two parts, one combinatorial and one group theoretic. We now overview both. First of all, in order to define the combinatorial part of the construction (referred to as the combinatorial frame in \cite{1205}) we need:
	
	\begin{itemize}
	\item $\mathbf{K}^{\mrm{eq}}$ is the class of models $N$ in a vocabulary $\{\mathfrak{E}_0, \mathfrak{E}_1, \mathfrak{E}_2\}$ such that each $\mathfrak{E}^N_i$ is an equivalence relation and $\mathfrak{E}^N_2$ is the equality relation. We use the symbol $\mathfrak{E}_i$ to avoid confusions, as the symbol $E_i$ also appears elsewhere.
	\item\label{universal_model} $M$ is the countable homogeneous universal model in $\mathbf{K}^{\mrm{eq}}$.
	\end{itemize}
	
	Practically, structures in $\mathbf{K}^{\mrm{eq}}$ are models of the theory of two equivalence relations, naming equality with $\mathfrak{E}^M_2$ is just a useful technical convenience which helps in some passages from \cite{1205}, that will also be used at the end of the proof of Theorem~\ref{second_theorem}.

\medskip 
	
	Now, a combinatorial frame is an object $\mathfrak{m}(M) = \mathfrak{m} = (X^\mathfrak{m}, \bar{X}^\mathfrak{m}, \bar{f}^\mathfrak{m}, \bar{E}^\mathfrak{m}) = (X, \bar{X}, \bar{f}, \bar{E})$ subject to several technical conditions, in particular $\bar{E}^{\mathfrak{m}} = \bar{E} = (E_n : 0 < n < \omega) = (E^{\mathfrak{m}}_n : 0 < n < \omega)$, and, for $0 < n < \omega$, 
	$E_n$ is an equivalence relation defined on injective $n$-sequences from $X$ (denoted as $\mrm{seq}_n(X)$). Although not made explicit here, all depends on $M$, as it depends on partial maps acting on $X$ given by $\bar{f} = (f_{\bar{g}} : \bar{g} \in \mathcal{G}_*)$, where $\mathcal{G}_*$ is made of sequences of partial isomorphisms of the universal model $M$. Furthermore, the set $X$ on which the partial maps $f_{\bar{g}}$'s act is partitioned into infinite pieces as $\bar{X} = (X'_s : s \in M)$. This allow us to define for every $\mathcal{U} \subseteq M$ the set $X_{\mathcal{U}} = \bigcup\{X'_s : s \in \mathcal{U} \}$.
These are the essential pieces of $\mathfrak{m}$.

\medskip 

	Now, given a combinatorial frame $\mathfrak{m}$ as above, we define a group $G_1 = G_1[\mathfrak{m}]$. This group will be some sort of universal model for our Borel reduction of $\mathbf{K}^{\mrm{eq}}_\omega$ into $\mrm{TFAB}_\omega$. Crucially, our group $G_1$ will have as basis (in the sense of abelian group theory) the set $X$ from the combinatorial frame $\mathfrak{m}= (X, \bar{X}, \bar{f}, \bar{E})$. The group $G_1$ encodes the combinatorial frame $\mathfrak{m}$ in a sophisticated manner, via divisibility conditions on elements of $G_1$. Now, given a set $\mathcal{U} \subseteq M$ we can consider the subset $X_\mathcal{U}$ of the basis $X$ and with it the group (recall the notation from \ref{def_pure}):
	$$G_{(1, \mathcal{U})}[\mathfrak{m}] = G_{(1, \mathcal{U})}[\mathfrak{m}(M)] = G_{(1, \mathcal{U})} = \langle y: y \in X_u, u \in \mathcal{U} \rangle^*_{G_1} = \langle X_\mathcal{U} \rangle^*_{G_1}.$$
	
	Essentially, the Borel reduction from \cite{1205} is the map $\mathcal{U} \mapsto G_{(1, \mathcal{U})}$.  In order to prove our results, we need a final piece of notation. Recall that the equivalence relations $\mathfrak{E}_i$ (for $i \in \{0, 1, 2\}$) are defined on the universal model $M$, while the group $G_1$ has as basis elements from $X = \bigcup\{X_s : s \in M \}$. We ``translate'' the equivalence relations $\mathfrak{E}^M_i$ on $M$ to equivalence relations $\mathcal{E}_i$ on $X$ as follows:

	\begin{definition}\label{the_auxiliary_eq} For $i < 3$, let:
	$$\mathcal{E}_i = \{ (x, y) : \text{ for some } (a, b) \in \mathfrak{E}^M_i, x \in X'_{a} \text{ and } y \in X'_{b}\}.$$
\end{definition}
	
This ends our overview of the construction from \cite{1205}. Evidently, a proper understanding of the details of the construction require the reader to refer to \cite{1205}, but we see no other way to explain the details of the construction without essentially reproducing the first \mbox{sections of \cite{1205} (as was done in a previous version of this paper).}

\section{Faithfulness}\label{faith_sec}



By an interpretation $\Gamma$ we mean as in \cite[pg. 212]{hodges}, so in particular we require the existence of objects $\partial_\Gamma$, $\phi_\Gamma$'s, and $f_\Gamma$ as there. In particular, the formula $\partial_\Gamma$ is referred to as the domain of the interpretation $\Gamma$. If all the formulas involved in the interpretation are $\mathfrak{L}_{\aleph_1, \aleph_0}$-formulas, then we talk of an $\mathfrak{L}_{\aleph_1, \aleph_0}$-interpretation. By an interpretation of a class of structures into another we mean as in \cite[Sec.~5.4(b)]{hodges}. 

	\begin{notation}\label{the_pure_formulas_notaiton}
	\begin{enumerate}[(1)]
	\item By $\mathfrak{L}^{\mrm{pure}}_{\aleph_1, \aleph_0}(\tau_{\mrm{AB}})$-interpretation we mean an $\mathfrak{L}_{\aleph_1, \aleph_0}$-interpretation in the language of abelian groups $\tau_{\mrm{AB}} = \{0, +, -\}$ which uses formulas in the closure of the following formulas by negation and countable conjunctions:
	$$\{p^m \, \vert \, x, \; p^m \, \vert \, (x-y), \; nx = ky, \; x = y : p \in \mathbb{P}, \; m, n, k < \omega\}.$$
	\item Below by ``definable'' we mean definable by a formula as in \ref{the_pure_formulas_notaiton}(1).
\end{enumerate}	 
\end{notation}

	\begin{fact}\label{purity_fact_interpretation} If $\varphi(\bar{x}) \in \mathfrak{L}^{\mrm{pure}}_{\aleph_1, \aleph_0}(\tau_{\mrm{AB}})$ and $G \leq_* H \in \mrm{AB}$, for $\bar{a} \in G^{\mrm{lg}(\bar{x})}$ we have that:
$$G \models \varphi(\bar{a}) \; \Leftrightarrow \; H \models \varphi(\bar{a}).$$
\end{fact}

\begin{definition} Let $X$ and $G_1$ be as in Section~\ref{S3}. For $a \in G_1$ we let:
	$$\mathbb{P}_a = \{p \in \mathbb{P}: p^\infty \vert \, a\}.$$
\end{definition}

	\begin{cclaim}\label{intepretation_claim} Let $\mathbf{B}: \mathbf{K}^{\mrm{eq}}_\omega \rightarrow \mrm{TFAB}_\omega$ be as in Proof of Main Theorem of \cite{1205}. 
	\begin{enumerate}[(1)]
	\item There is an $\mathfrak{L}^{\mrm{pure}}_{\aleph_1, \aleph_0}(\tau_{\mrm{AB}})$-interpretation of $\mathbf{K}^{\mrm{eq}}_\omega$ into $\{\mathbf{B}(N) : N \in \mathbf{K}^{\mrm{eq}}_\omega\}$.
	\item There is $\psi_* \in \mathfrak{L}_{\aleph_1, \aleph_0}(\tau_{\mrm{AB}})$ such that for every countable abelian group $G$ we have that $G \models \psi_*$ if and only if there is $N \in \mathbf{K}^{\mrm{eq}}_\omega$ such that $G \cong \mathbf{B}(N)$.
\end{enumerate}
\end{cclaim}

	\begin{proof} We prove (1). Let $M$, $\mathfrak{m}$, $X$ and $G_1$ be as in Section~\ref{S3}, and $G = G_{\mathcal{U}} = G_{(1, \mathcal{U})}$, for $\mathcal{U} \subseteq M$. Notice that, although we fix $\mathcal{U} \subseteq M$ and $G$ for almost all the proof, all the formulas that we define below \underline{do not} depend on $\mathcal{U}$.
	\begin{enumerate}[$(\star_1)$]
	\item Let $\mathcal{E}_\star = \{(a, b) \in G_1 : a \neq 0 \neq b \wedge ma = nb, \text{ for some } m, n \in \mathbb{Z}^+ \}$.
	\end{enumerate}
	\begin{enumerate}[$(\star_{2})$]
	\item $\mathcal{E}_\star$ is a definable equivalence relation.
	\end{enumerate}
	\begin{enumerate}[$(\star_3)$]
	\item From here until $(\star_7)$, fix $\hat{x} \in X$.
	\end{enumerate}
	\begin{enumerate}[$(\star_4)$]
	\item We define a formula $\psi^{\hat{x}}(v)$ (so $v$ is a free variable) saying the following:
	\begin{enumerate}
	\item $v$ is $p^\infty$-divisible for every prime $p \in \mathbb{P}_{\hat{x}}$.
	\item $v$ is not $p^\infty$-divisible for every $p \in \mathbb{P}_{\sum_{\ell < k} q_\ell x_\ell}$, where:
	\begin{enumerate}
	\item $k \geq 2$;
	\item $(x_\ell : \ell < k) \in \mrm{seq}_k(X)$;
	\item $\bar{q} \in (\mathbb{Z}^+)^k$.
	\end{enumerate}
	\item $v \neq 0$ (this actually follows from (b)).
	\end{enumerate}
	\end{enumerate}
	\begin{enumerate}[$(\star_5)$]
	\item If $a \in G$, $y \in \hat{x}/E^{\mathfrak{m}}_1 \cap X_\mathcal{U}$ and $a \in y/\mathcal{E}_\star$, then $G \models \psi^{\hat{x}}(a)$.
	\end{enumerate}
	The fact that $(\star_5)$ holds is easy to see, recalling that $G \leq_* G_1$.
	\begin{enumerate}[$(\star_6)$]
	\item \begin{enumerate}[(a)]
	\item If $a \in G$ and $|\mrm{supp}(a)| \geq 2$, then $G \models \neg \psi^{\hat{x}}(a)$;
	\item If $y \in X_\mathcal{U}$, $a = qy \in G$ and $y \notin \hat{x}/E^{\mathfrak{m}}_1$, then $G \models \neg \psi^{\hat{x}}(a)$.
	\end{enumerate}
\end{enumerate}
Clauses (a) and (b) can be proved arguing as in the proof of \cite[Lemma~4.8]{1205}, in particular concerning clause (b) cf. the argument given in $(*_0)$ from the proof of \cite[Lemma~4.8]{1205}.
\begin{enumerate}[$(\star_7)$]
	\item If $a \in G$, then $G \models \psi^{\hat{x}}(a)$ iff $a \in \bigcup \{y/\mathcal{E}_\star : y \in \hat{x}/E^{\mathfrak{m}}_1\}$.
\end{enumerate}
This is by $(\star_5)$ and $(\star_6)$.
\begin{enumerate}[$(\star_8)$]
	\item If $\hat{x} \in X$, then for every $\mathcal{U}_1 \subseteq M$ we have:
	$$\hat{x} \in X_{\mathcal{U}_1} \; \Leftrightarrow \; \hat{x}/\mathcal{E}_\star \subseteq G_{(1, \mathcal{U}_1)}.$$
\end{enumerate}
	\begin{enumerate}[$(\star_9)$]
	\item For $\hat{x}, \hat{y} \in X$, we define a formula $\psi^{\hat{x}, \hat{y}}(v)$ saying the following: 
	\begin{enumerate}[(a)]
	\item $v$ is $p^\infty$-divisible for every prime $p \in \mathbb{P}_{\hat{x}-\hat{y}}$;
	\item $v$ is not $p^\infty$-divisible when for some $x \neq y \in X$ we have $(x, y) \notin (\hat{x}, \hat{y})/E^{\mathfrak{m}}_2$ and $p \in \mathbb{P}_{x-y}$.
	\end{enumerate}
\end{enumerate}
\begin{enumerate}[$(\star_{10})$]
	\item If $G \models \psi^{\hat{x}}(a) \wedge \psi^{\hat{y}}(b) \wedge \psi^{\hat{x}, \hat{y}}(a-b)$, then for some $x_1, y_1$ and $q \in \mathbb{Q}^+$ we have:

	\begin{enumerate}[(a)]
	\item $x_1, y_1 \in X_{\mathcal{U}}$;
	\item $a = qx_1 \in G$ and $b = qy_1 \in G$;
	\item $(x_1, y_1) E^{\mathfrak{m}}_2 (\hat{x}, \hat{y})$, so $x_1 \in \hat{x}/E^{\mathfrak{m}}_1 \cap X_{\mathcal{U}}$, $y_1 \in \hat{y}/E^{\mathfrak{m}}_1 \cap X_{\mathcal{U}}$.
	\end{enumerate}
\end{enumerate}
We show that $(\star_{10})$ holds. The existence of $x_1, y_1 \in X_{\mathcal{U}}$ such that $a \in x_1/\mathcal{E}_\star \cap X_{\mathcal{U}}$ and $b \in y_1/\mathcal{E}_\star \cap X_{\mathcal{U}}$  holds by $(\star_{7})$ and the assumption. Furthermore, as $G \models \psi^{\hat{x}}(a) \wedge \psi^{\hat{y}}(b) \wedge \psi^{\hat{x}, \hat{y}}(a-b)$, then necessarily $x_1, y_1 \in X_{\mathcal{U}}$.
Let now $a = q_1 x_1$ and $b = q_2 y_1$, for $q_1, q_2 \in \mathbb{Q}^+$. For the sake of contradiction suppose that $q_1 \neq q_2$. As $G \models \psi^{\hat{x}, \hat{y}}(a - b)$ we know that for every $p \in \mathbb{P}_{\hat{x} - \hat{y}}$ we have that $G \models p^\infty \, \vert \, (q_1x_1 - q_2 y_1)$. Let $q \in \mathbb{Z}^+$ be such that $qq_1, qq_2 \in \mathbb{Z}$ and let $p \in \mathbb{P}_{\hat{x} - \hat{y}}$ be $> |qq_1| + |qq_2|$.
Now, 
we can find $n$ and $(q^\ell, x^\ell, y^\ell : \ell < n)$ such that $x^\ell, y^\ell \in X_{\mathcal{U}}$, $q_\ell \in \mathbb{Z}^+$, $q^\ell (x^\ell - y^\ell) \in G$, $q \in \mathbb{Z}^+$ and $(x^\ell, y^\ell) \in (\hat{x}, \hat{y})/E^{\mathfrak{m}}_2$ and we have the following:
	$$q(q_1 x_1 - q_2 y_1) = \sum_{\ell < n}q^\ell (x^\ell - y^\ell) \;\;\mrm{mod}(\mathbb{Q}_pG_0 \cap G_1).$$
But analyzing the equation above we have that the sum of the coefficients on the LHS is $q(q_1 - q_2) \neq 0$ (recall that by assumption $q_1 \neq q_2$), whereas on the RHS it is zero, a contradiction. Finally, the fact that $(x_1, y_1) E^{\mathfrak{m}}_2 (\hat{x}, \hat{y})$ is by (b) of $(\star_{9})$.
\begin{enumerate}[$(\star_{11})$]
	\item Recalling \ref{the_auxiliary_eq}, for $i = 0, 1, 2$, let $\chi'_i(a, b)$ be the formula:
	$$\bigvee \{\psi^{\hat{x}}(a) \wedge \psi^{\hat{y}}(b) \wedge \psi^{\hat{x}, \hat{y}}(a-b): \hat{x}, \hat{y} \in X \text{ and } G_1 \models \hat{x} \mathcal{E}_i \hat{y} \}.$$
\end{enumerate}
	\begin{enumerate}[$(\star_{12})$]
\item For $i = 0, 1, 2$, let $\chi_i(a, b)$ be the formula:
	$$\exists a_1, b_1 (a \mathcal{E}_\star a_1 \wedge b \mathcal{E}_\star b_1 \wedge \chi'_i(a_1, b_1)).$$
\end{enumerate}
\begin{enumerate}[$(\star_{13})$]
	\item For $\mathcal{U} \subseteq M$, $a, b \in G = G_{(1, \mathcal{U})}$ and $i < 3$, we have that TFAE:
	\begin{enumerate}
	\item $G \models \chi_i(a, b)$;
	\item for some $\mathcal{E}_i$-equivalence class $Y \subseteq X_{\mathcal{U}}$ we have $a, b \in \bigcup \{x/\mathcal{E}_\star : x \in Y\}$.\end{enumerate}
	\end{enumerate}
	We show that $(\star_{13})$ holds. The interesting direction is ``(a) implies (b)''. So assume that $G \models \chi_i(a, b)$, then there are $a_1 \in a/\mathcal{E}_\star$ and $b_1 \in b/\mathcal{E}_\star$ such that $G \models \chi'_i(a_1, b_1)$. Hence, for some $\hat{x}, \hat{y} \in X$ we have that:
	\begin{enumerate}[(i)]
	\item  $G \models \psi^{\hat{x}}(a_1) \wedge \psi^{\hat{y}}(b_1) \wedge \psi^{\hat{x}, \hat{y}}(a_1 - b_1)$;
	\item $\hat{x} \mathcal{E}_i \hat{y}$.
	\end{enumerate}
	Now, by $(\star_{10})$, for some $x \in \hat{x}/E^{\mathfrak{m}}_1 \cap X_{\mathcal{U}}$, $y \in \hat{y}/E^{\mathfrak{m}}_1 \cap X_{\mathcal{U}}$ and $q \in \mathbb{Q}^+$ we have that $a_1 = qx, b_1 = qy \in G_1$ and $(x, y) E^{\mathfrak{m}}_2 (\hat{x}, \hat{y})$. But by \cite[Claim4.11(1)]{1205} we have that $(x \mathcal{E}_i y)$ iff $(\hat{x} \mathcal{E}_i \hat{y})$, and so by (ii) above we are done. This is enough for our purposes as we can now interpret a model isomorphic to $M \restriction \mathcal{U}$ in $G_{(1, \mathcal{U})} = G$ in the following manner:
	\begin{enumerate}[$(\star_{14})$]
	\item
	\begin{enumerate}[(a)]
	\item the domain of the interpretation is $\{a \in G : G \models \chi_2(a, a)\}$ (cf. what was said in the beginning of the present section, so this correspond to the $\partial_\Gamma$ from \cite[pg. 212]{hodges});
	\item equality is interpreted as $\chi_2(a, b)$ (recall that $\mathfrak{E}_2$ is $=$ on $M$, cf. the beginning of Section~\ref{S3});
	\item we interpret $\mathfrak{E}_i$ as $\chi_i(a, b)$.
	\end{enumerate}
	\end{enumerate}
This concludes the proof of (1). We now prove (2).
\begin{enumerate}[$(*_1)$]
\item For $s \in M$ and $\hat{x} \in X'_s$, let $\varphi^*_{(s, \hat{x})}(v) \in \mathfrak{L}_{\aleph_1, \aleph_0}(\tau_{\mrm{AB}})$ say:
\begin{enumerate}
	\item $\chi_2(a, a)$;
	\item for $p \in \mathbb{P}_{\hat{x}}$, $p^\infty \vert a$.
\end{enumerate}
\end{enumerate}
\begin{enumerate}[$(*_2)$]
\item For $s \in M$ and $\hat{x} \in X'_s$, let $\psi^*_{(s, x)}(v) \in \mathfrak{L}_{\aleph_1, \aleph_0}(\tau_{\mrm{AB}})$ be such that, for every abelian group $C$ and $a \in C$, $C \models \psi^*_{(s, \hat{x})}(a)$ if and only if:
\begin{enumerate}
	\item $C \in \mrm{TFAB}$;
	\item $C \models \varphi^*_{(s, \hat{x})}(a)$;
	\item $C$ is the pure closure of the subgroup generated by:
	$$A = \{b \in C : C \models \chi_*(b, a)\},$$
where $\chi_*(b, a)$ is the formula:
	$$\bigvee \{\psi^{\hat{x}}(a) \wedge \psi^{\hat{y}}(b) \wedge \psi^{\hat{x}, \hat{y}}(a-b): \hat{y} \in X \};$$
	\item the set $A$ is (linearly) independent;
	\item if $b \in A$ and $C \models \psi^{\hat{x}}(a) \wedge \psi^{\hat{y}}(b) \wedge \psi^{\hat{x}, \hat{y}}(a-b)$, then for every $b' \in C$ we have that $C \models \psi^{\hat{x}}(a) \wedge \psi^{\hat{y}}(b') \wedge \psi^{\hat{x}, \hat{y}}(a-b')$ implies that $b = b'$;
	\item moreover, if $C \models \psi^{\hat{x}}(a) \wedge \psi^{\hat{y}}(b) \wedge \psi^{\hat{x}, \hat{y}}(a-b)$, then $b$ and $\hat{x}$ determine $\hat{y}$ uniquely;
	\item For $i = 0, 1, 2$, $\chi_i(u, v)$ is an equivalence relation on $A$ with infinitely many equivalence classes, and, for $i = 0, 1$, each equivalence class is infinite;
	\item if $k < \omega$, $(b_\ell : \ell < k) \in A^k$ is without repetitions and, $q_\ell \in \mathbb{Z}^+$, for $\ell < k$,  and $p$ is a prime such that $p$ does not divide $q_\ell$ for some $\ell < k$, \underline{then}, letting $\hat{y}_\ell \in X$ be such that $\psi^{\hat{x}, \hat{y}}(a-b_\ell)$, the following are equivalent:
	\begin{enumerate}[($\cdot_1$)]
	\item $p \in \mathbb{P}_{\sum_{\ell < k} q_\ell y_\ell}$;
	\item $p \, \vert \sum \{q_\ell b_\ell : \ell < k \}$;
	\item $p^\infty \, \vert \, \{q_\ell b_\ell : \ell < k\}$;
	\end{enumerate}
	\item if $t \in M$ and $\hat{z}_1, \hat{z}_2 \in X'_t$, then ($\cdot_1$) $\Leftrightarrow$ ($\cdot_2$), where:
	\begin{enumerate}[($\cdot_1$)]
	\item $C \models \psi^{\hat{x}}(a) \wedge \psi^{\hat{z}_1}(b_1) \wedge \psi^{\hat{x}, \hat{z}_1}(a-b_1)$, for some $b_1 \in C$;
	\item $C \models \psi^{\hat{x}}(a) \wedge \psi^{\hat{z}_2}(b_2) \wedge \psi^{\hat{x}, \hat{z}_2}(a-b_2)$, for some $b_2 \in C$;
	\end{enumerate}
	\item $A$ is a maximal independent subset of $C$.
\end{enumerate}
\end{enumerate}
\begin{enumerate}[$(*_3)$]
\item If $s \in M$, $\hat{x} \in X'_s$ and $s \in \mathcal{U} \subseteq M$  with $\mathcal{U}$ infinite, then $G_{\mathcal{U}} \models \psi^*_{(s, \hat{x})}(\hat{x})$.
\end{enumerate}
[Why? By construction.]
\begin{enumerate}[$(*_4)$]
\item For any countable abelian group $C$, $s \in M$ and $\hat{x} \in X'_s$ we have that: $C \models \exists u \psi^*_{(s, \hat{x})}(u)$ \underline{if and only if}, for some infinite $\mathcal{U} \subseteq M$ with $s \in \mathcal{U}$, $C \cong \mathbf{B}(M \restriction \mathcal{U})$.
\end{enumerate}
We prove $(*_4)$. The right-to-left implication is by $(*_3)$. Concerning the other implication, suppose that $C \models \exists u \psi^*_{(s, \hat{x})}(u)$ and let $a_* \in C$ be such that $C \models \psi^*_{(s, \hat{x})}(a_*)$. Now, recalling the definition of $\psi^*_{(s, \hat{x})}(u)$ from  $(*_2)$, for $b \in A$, let:
$$Y_b = \{\hat{y} \in X : \psi^{\hat{x}}(a_*) \wedge \psi^{\hat{y}}(b) \wedge \psi^{\hat{x}, \hat{y}}(a_*-b)\}$$
\begin{enumerate}[$(*_{4.1})$]
\item If $b \in A$, then $|Y_b| = 1$.
\end{enumerate}
[Why? By $(*_{2})$.]
\begin{enumerate}[$(*_{4.2})$]
\item Let $g_0$ be the function with domain $A$ such that $g_0(b) \in Y_b$, for $b \in A$.
\end{enumerate}
[Why $g_0$ exists? By $(*_{4.1})$.]
\begin{enumerate}[$(*_{4.3})$]
\item 
\begin{enumerate}[(a)]
	\item Let $Y = \mrm{ran}(g)$;
	\item $\mathcal{U} = \{t \in M : Y \cap X'_t \neq \emptyset\}$.
\end{enumerate}
\end{enumerate}
\begin{enumerate}[$(*_{4.4})$]
\item $Y = \bigcup \{X'_t : t \in \mathcal{U}\}$.
\end{enumerate}
[Why $(*_{4.4})$? The inclusion $\subseteq$ holds by $(*_{4.3})$(b). The other inclusion is by~$(*_{2})$(i).]
\begin{enumerate}[$(*_{4.5})$]
\item $Y = X_\mathcal{U}$.
\end{enumerate}
[Why? By $(*_{4.4})$ and the choice of $X_{\mathcal{U}}$.]
\begin{enumerate}[$(*_{4.6})$]
\item There is an isomorphism $g_0 \subseteq g_1: \langle A \rangle_C \cong G_{(0, \mathcal{U})} = \sum \{\mathbb{Z}y : y \in Y = X_{\mathcal{U}}\}$.
\end{enumerate}
[Why? As $A$ is independent by $(*_{2})$(d).]
\begin{enumerate}[$(*_{4.7})$]
\item There is an isomorphism $g_1 \subseteq g_2: \langle A \rangle^*_C \cong G_{(1, \mathcal{U})} = \langle Y \rangle^*_{G_{(1, \mathcal{U})}}$.
\end{enumerate}
[Why? By $(*_{2})$(h) and $(*_{4.5})$.]
\newline Hence, $(*_4)$ holds indeed.
\begin{enumerate}[$(*_5)$]
\item For any countable abelian group $C$, we have that 
$$C \models \bigvee \{ \exists u \psi^*_{(s, \hat{x})}(u) : s \in M, \hat{x} \in X'_s\} := \psi_*$$
if and only if, for some infinite $\mathcal{U} \subseteq M$, $C \cong \mathbf{B}(M \restriction \mathcal{U})$.
\end{enumerate}
\end{proof}

	\begin{proof}[Proof of Theorem~\ref{first_theorem}] This follows from \ref{intepretation_claim} together with the well-known fact that $\mathbf{K}^{\mrm{eq}}_\omega$ is faithfully Borel complete (folklore, see also the proof of Fact~\ref{embed_2equiv}).
\end{proof}

\section{Pure embeddability is a complete analytic quasi-order}

\begin{fact}\label{embed_2equiv} There is a Borel map $\mathbf{B}$ from $\mrm{Graph_\omega}$ into $\mathbf{K}^{\mrm{eq}}_\omega$ such that we have:
$$H_1 \text{ embeds into } H_2 \; \Leftrightarrow \; \mathbf{B}(H_1) \text{ embeds into }\mathbf{B}(H_2).$$
\end{fact}

\begin{proof} This is folklore but we add details for the benefit of the reader. For a graph $H = (H, R^H)$ with domain $\subseteq \omega$ we define a model $M = \mathbf{B}(H)$ of the theory of two equivalence relations with set of elements $\omega \cup \omega \times \omega$ defining $E^M_1, E^M_2$ as follows:
	\begin{enumerate}[(1)]
	\item $E^M_1$ partitions $M$ into the sets $X_n = \{n\} \cup \{ (n, m) : m< \omega \}$, for $n < \omega$;
	\item $E^M_2 = \{(n, m) : n, m < \omega\} \cup \{((n, m), (m, n)) : nR^H m\} \cup =_M$.
	\end{enumerate}
It is easy to see that $E^M_1$ and $E^M_2$ define equivalence relations on $M$ and so $M = \mathbf{B}(H) \in \mathbf{K}^{\mrm{eq}}_\omega$.
Notice now that $H$ is first-order interpretable in $\mathbf{B}(H)$ as follows:
\begin{enumerate}[(A)]
	\item the domain of the interpretation is the set of elements $\varphi_0(x)$ such that $x/E_2$ has at least three elements and equality is interpreted as equality;
	\item the edge relation on $\varphi_0(M)$ is defined as $\varphi_R(x, y)$ iff there are $x_1$ and $y_1$ s.t.:
	$$x E_1 x_1 \wedge x_1 E_2 y_1 \wedge y_1 E_1 y.$$
\end{enumerate}
It is then easy to see that the Borel map $H \mapsto \mathbf{B}(H)$ is as wanted.
\end{proof}

	\begin{proof}[Proof of Theorem~\ref{second_theorem}] First of all, notice that with a slight abuse of notation (but not a problematic one) we consider models with domain $\subseteq \omega$ instead of simply $\omega$. Notice now that the Borel map $\mathbf{B}$ from the proof of the Main Theorem from \cite{1205} is such that for $H_1, H_2 \in \mathbf{K}^{\mrm{eq}}_\omega$ we have that:
	$$H_1 \text{ embeds into } H_2 \; \Leftrightarrow \; \mathbf{B}(H_1) \text{ embeds purely into }\mathbf{B}(H_2),$$
and so by Fact~\ref{embed_2equiv} we are done as it was proved in \cite{LouRos} that embeddability between countable graphs is a complete analytic quasi-order. Finally, it is easy to see that all the torsion-free abelian groups in our construction are elementary equivalent to $\mathbb{Z}^{(\omega)}$ and it is well-known that elementary embeddability among models of a complete theory of $\mrm{TFAB}$ corresponds to pure embeddability, see e.g. \cite[Appendix~6.2]{hodges}.
\end{proof}

\end{document}